\documentclass[a4paper,twoside,12pt,reqno]{amsart}
\usepackage[headings]{fullpage}
\usepackage{amsmath} \usepackage{amssymb} \usepackage{amsopn}
\usepackage{amsthm} \usepackage{amsfonts} \usepackage{mathrsfs}
\usepackage{mathtools}
\usepackage{stmaryrd}
\usepackage{manfnt} \usepackage{mathdots}
\usepackage[OT2, T1]{fontenc}
\usepackage[dvips]{graphicx} \usepackage[dvipsnames,usenames]{xcolor}
\usepackage[all]{xy}
\usepackage{tikz} \usetikzlibrary{decorations.markings,matrix,arrows,chains,calc}
\usepackage{dsfont}
\usepackage[latin1]{inputenc}  
\usepackage{ifthen}
\usepackage{marvosym}
\usepackage{wasysym}
\usepackage{marginnote}
\usepackage[pdfpagelabels, pdftex,bookmarksopen,bookmarksdepth=2]{hyperref}
\hypersetup{
  pdftitle={},
  pdfauthor={},
  pdfsubject={},
  pdfkeywords={},
  colorlinks=true,    
  linkcolor=BrickRed,
  citecolor=OliveGreen,     
  filecolor=Blue,      
  urlcolor=Blue,       
  breaklinks=true,
  bookmarksopen=true,
  bookmarksnumbered=true,
  pdfpagemode=UseOutlines,
  plainpages=false}

\newcommand{\bC}{{\mathbb C}}

\newcommand{\bN}{{\mathbb N}}

\newcommand{\bQ}{{\mathbb Q}}

\newcommand{\bZ}{{\mathbb Z}}

\newcommand{\dO}{{\mathcal O}}
\newcommand{\dP}{{\mathcal P}}

\DeclareSymbolFont{cyrletters}{OT2}{wncyr}{m}{n}
\DeclareMathSymbol{\Sha}{\mathalpha}{cyrletters}{"58}

\newcommand{\one}{\mathbf{1}}

\newcommand{\xyinj}{\ar@{^(->}}

\def\10{{\overrightarrow{10}}}
\def\01{{\overrightarrow{01}}}

\newcommand{\ep}{\varepsilon}

\newcommand{\ph}{\varphi}

\newtheorem{thm}{Theorem}[section]

\newtheorem{prop}[thm]{Proposition}
\newtheorem{lem}[thm]{Lemma}

\newtheorem{thmABC}{Theorem}

\theoremstyle{definition}

\theoremstyle{remark}
\newtheorem{rmk}[thm]{Remark}

\newtheorem{ex}[thm]{Example}

\numberwithin{equation}{section}

\usepackage{enumitem}

\let\origmaketitle\maketitle
\def\maketitle{
  \begingroup
  \def\uppercasenonmath##1{} 
  \origmaketitle
  \endgroup
}

\newcommand{\qZ}[1]{\dO_{\!-#1}}

\newcounter{step}[thm]

\begin{document}
\title[An algebraic approach to count the representations by $x^2+ay^2$]{An algebraic approach to count the number of representations of an integer by the quadratic form $x^2+ay^2$ for certain values of $a$}
\author{Thanathat Dechakulkamjorn}
\address[Thanathat Dechakulkamjorn]{Department of Mathematics and Computer Science, Chulalongkorn University, Bangkok, Thailand}
\email{ball09631@gmail.com}

\author{Nithi Rungtanapirom}
\address[Nithi Rungtanapirom]{Department of Mathematics and Computer Science, Chulalongkorn University, Bangkok, Thailand}
\email[corresponding author]{nithi.r@chula.ac.th}

\begin{abstract}
By considering the norm of elements in the ring of integers in $\bQ(\sqrt{-a})$, we give an algebraic approach to count the number of integral solutions of diophantine equations of the form $x^2+ay^2=n$ where $a$ is a Heegner number or $a=27$.
\end{abstract}


\maketitle

\tableofcontents


\section{Introduction}
The theory of representation of integers by binary quadratic forms has been studied for a long time. One of interesting problems is to count the number of representations of a fixed integer by a given binary quadratic form. Dirichlet's work \cite{Dirichlet1871} dealt with a variation of this problem, namely representations by the collection of reduced binary quadratic forms of a given discriminant. Based on this work, Hall \cite{newman} derived a formula for the case that each genus of binary quadratic forms of the given discriminant consists of exactly one reduced form. Further investigations on the number of representations by certain single binary quadratic forms are based on an analytic approach using Epstein zeta functions, theta series or Dirichlet series among the others. The latest investigations have been done by Kaplan and Williams \cite{Kaplan2004OnTN}; Sun and Williams \cite{Sun2006OnTN}; Berkovich and Yesilyurt \cite{berkovich-ramanujan}; Bagis and Glaser \cite{bagis-glasser}; etc.

This paper focuses on an algebraic approach for this problem in several cases. The key tool for our approach is the ring of integers $\qZ{a}$ in $\bQ(\sqrt{-a})$. Such an approach has been employed, for instance, in \cite{busenhart-halbeisen-hungerbuehler} to find an explicit formula for primitive solutions of the equation $x^2+y^2=n$ using Gaussian integers. More precisely, the expression $x^2+ay^2$ should be interpreted as the norm of $x+y\sqrt{-a}$, so that one may first consider instead counting the number of elements of $\qZ{a}$ of the given norm. As we shall see in Proposition \ref{prop:count-given-norm}, this works well if $a$ is a \textbf{Heegner number}, or equivalently, if $\qZ{a}$ is a unique factorization domain. It is known that a Heegner number is one of the following numbers:
\[
1, 2, 3, 7, 11, 19, 43, 67, 163,
\]
see \cite{stark-heegner}. The slight difficulty for $a\geq 3$ is $\qZ{a}$ also contains linear combinations of $1$ and $\sqrt{-a}$ with half-integer coefficients. Consequently, a more careful investigation is necessary. It turns out that the case $a=27$ can also be done by this approach. To summarize, we obtain the following results:

\begin{thmABC}
Let $n$ be a natural number. For each natural number $a$, let $X(n,a)$ denote the set of integral solutions to the equation $x^2+ay^2=n$.
\begin{enumerate}[label=\textup{(\Roman*)}]
\item \textup{(Theorem \ref{thm:Xan-1-2})} For $a=1,2$, we have
    \[
    |X(n,1)| = 4 \sum_{c\mid n} \left(\frac{-4}{c}\right) \quad \text{and} \quad |X(n,2)| = 2 \sum_{c\mid n} \left(\frac{-2}{c}\right).
    \]
\item \textup{(Theorem \ref{thm:Xan-3})} For $a=3$, we have the following results:
    \begin{enumerate}[label=\textup{(\alph*)}]
    \item If $n$ is even, then $\displaystyle |X(n,3)| = 6\displaystyle\sum_{c|n} \left ( \frac{c}{3} \right )$.
    \item If $n$ is odd, then $\displaystyle |X(n,3)| = 2\displaystyle\sum_{c|n} \left ( \frac{c}{3} \right )$.
    \end{enumerate}
\item \textup{(Theorem \ref{thm:Xan-7})} For $a=7$, we have the following results:
    \begin{enumerate}[label=\textup{(\alph*)}]
    \item If $4\mid n$, then $\displaystyle |X(n,7)| = 2\displaystyle\sum_{c|\frac{n}{4}} \left ( \frac{c}{7} \right )$.
    \item If $n$ is even but $4\nmid n$, then $X(n,7)=\emptyset$.
    \item If $n$ is odd, then $\displaystyle |X(n,7)| = 2\displaystyle\sum_{c|n} \left ( \frac{c}{7} \right )$.
    \end{enumerate}
\item \textup{(Theorem \ref{thm:Xan-11})} For every Heegner number $a\geq 11$, we have the following results:
    \begin{enumerate}[label=\textup{(\alph*)}]
    \item If $n$ is even, then $\displaystyle |X(n,a)| = |Y(n,a)|$.
    \item If $n$ is odd, then
    \[
    |X(n,a)| = \frac13\left[1+2\cdot\frac{(\tau(n_q)|3)}{\tau(n_q)}\right]\cdot|Y(n,a)|,
    \]
    where $n_q$ denotes the product of all prime factors (including multiplicity) of $n$ which are quadratic residues modulo $a$ but not expressible as $x^2+ay^2$ for any integers $x,y$.
    \end{enumerate}
\item \textup{(Theorem \ref{thm:Xan-27})} For $a=27$, we have the following results:
    \begin{enumerate}[label=\textup{(\alph*)}]
    \item If $3 \mid n$, then $X(a,27)\neq \emptyset$ only if $9 \mid n$. In this case, we have 
    \[
    |X(n,27)| = |X(\tfrac{n}{9},3)|.
    \]
    \item If $3\nmid n$ but $2\mid n$, then
    \[ |X(n,27)| = \frac13|X(n,3)|. \]
    \item If $\gcd(6,n)=1$, then
    \[
    |X(n,27)| = \frac13\left[1+2\cdot\frac{(\tau(n_q)|3)}{\tau(n_q)}\right]\cdot|X(n,3)|,
    \]
    where $n_q$ denotes the product of all prime factors $q$ (including multiplicity) of $n$ such that $q\equiv1\pmod{3}$ and $2$ is not a cubic residue modulo $q$.
    \end{enumerate}
\end{enumerate}
\end{thmABC}

This paper is organized as follows: Section \ref{sec:count-norm} deals with a related problem, namely counting the elements of given norm. The problem of our interest will be discussed in Section \ref{sec:count-repr} except the cases $a=3,27$. These two cases will be discussed in Section \ref{sec:eisenstein}, where the ring of Eisenstein integers with primitive third roots of unity is involved.

\smallskip

\paragraph{Notation}
The ring of integers in $\bQ(\sqrt{-a})$, where $a$ is a square-free natural number, will be denoted by
\[
\qZ{a} = \begin{cases} \bZ[\sqrt{-a}] & \text{if $a\equiv 1,2\pmod{4}$,}\\ \bZ[\frac{-1+\sqrt{-a}}2] & \text{if $a\equiv 3\pmod{4}$.} \end{cases}
\]
We will also write $\lambda_a:=\frac{-1+\sqrt{-a}}{2}\in\qZ{a}$ if $a\equiv3\bmod{4}$. The conjugate of $z=x+y\sqrt{-a}\in\bQ(\sqrt{-a})$, where $x,y\in\bQ$, will be denoted by $\overline{z}:=x-y\sqrt{-a}$. The norm of $z$ will be denoted by $N(z):=z\overline{z}=x^2+ay^2$.

In order to count the number of representations, we introduce the following sets for natural numbers $a$ and $n$:
\begin{align*}
X(n,a) &:= \{(x,y)\in\bZ\times\bZ \ | \ x^2+ay^2=n\} \quad \text{and} \\
Y(n,a) &:= \left\{ z\in \qZ{a} \ \big| \ N(z)=n \right\}.
\end{align*}
Both sets are connected by the map
\[
\ph_{n,a}:X(n,a)\mapsto Y(n,a), \ (x,y)\mapsto x+y\sqrt{-a}
\]
as to be discussed in Sections \ref{sec:count-repr} and \ref{sec:eisenstein}. Finally, $\tau(n)$ denotes as usual the number of positive divisors of $n$.

\section{Counting the elements of given norm} \label{sec:count-norm}
In order to count the number of elements of $Y(n,a)$, recall the Kronecker symbol $\left(\frac{a}{n}\right)$ or $(a|n)$, which is a generalization of the Legendre symbol. It is multiplicative in the upper and lower arguments. Note that if $d$ is a square-free integer and $D$ is the discriminant of $\bQ(\sqrt{d})/\bQ$, the following holds for every prime number $p$:
\[
\left( \frac{D}{p} \right) = \begin{cases} 1 & \text{if $p$ splits in $\bQ(\sqrt{d})/\bQ$,} \\ 0 & \text{if $p$ is ramified in $\bQ(\sqrt{d})/\bQ$,} \\ -1 & \text{if $p$ is inert in $\bQ(\sqrt{d})/\bQ$.}\end{cases}
\]

\begin{prop} \label{prop:count-given-norm}
Let $a$ be a Heegner number. Furthermore, let $D$ be the discriminant of $\bQ(\sqrt{-a})/\bQ$. The following formula holds for all natural numbers $n$:
\begin{equation} \label{eqn:Yan-general}
|Y(n,a)| = |\qZ{a}^\times|\cdot \sum_{c\mid n} \left(\frac{D}{c}\right).
\end{equation}
In particular, if $a\equiv3\pmod4$, then
\begin{equation} \label{eqn:Yan-3mod4}
|Y(n,a)| = |\qZ{a}^\times|\cdot \sum_{c\mid n} \left(\frac{c}{a}\right).
\end{equation}
\end{prop}
Note that $|\qZ{a}^\times|=2$ for all square-free natural numbers $a$ except $a=1$, where $\qZ{1}^\times$ is the group of the fourth roots of unity, and $a=3$, where $\qZ{3}^\times$ is the group of the sixth roots of unity.
\begin{proof}
For simplicity, write $u:=|\qZ{a}^\times|$. Consider the function
\[
f : \mathbb{N} \rightarrow \mathbb{Q}, \ n\mapsto f(n) := \frac{1}{u} |Y(n,a)|.
\]
We want to show that $f$ is multiplicative. It is evident that $f(1)=1$. Now let $m,n \in \mathbb{N}$ be such that $\gcd(m,n) = 1$. Consider the mapping
\[
h:Y(m,a)\times Y(n,a) \to Y(mn,a), \ (z_1,z_2)\mapsto z_1z_2.
\]
This is well-defined since the norm on $\qZ{a}$ is multiplicative. Furthermore, the uniqueness of factorization in $\qZ{a}$ implies that each $z\in Y(mn,a)$ can be factored as product of two elements of norms $m$ and $n$ uniquely up to association. Hence $h$ is a $u$-to-one mapping, implying that
\[
|Y(m,a)|\cdot|Y(n,a)| = u\cdot|Y(mn,a)|.
\]
This proves the multiplicativity of $f$. Now we compute $|Y(a,p^k)|$ for each prime number $p$ and $k \in \mathbb{N}$ as follows:

\begin{enumerate}[leftmargin=4ex,labelwidth=8ex,label=\textsc{Case} \arabic*:,itemindent=6ex,itemsep=1ex]
\item $(D|p)=0$.\\
In this case, there is exactly one prime element $\pi\in\qZ{a}$ of norm $p$ up to association. This implies that every element in $\qZ{a}$ of norm $p^k$ is an associate of $\pi^k$. This implies that $|Y(a,p^k)|=u$.
\item $(D|p)=1$.\\
In this case, there are exactly two prime elements $\pi_1,\pi_2\in\qZ{a}$ of norm $p$ up to association. This implies that every element in $\qZ{a}$ of norm $p^k$ is an associate of $\pi_1^j\pi_2^{k-j}$ for some $j\in\{0,1,\ldots,k\}$. This implies that $|Y(a,p^k)|=u(k+1)$.
\item $(D|p)=-1$.\\
In this case, $p$ remains prime in $\qZ{a}$ and has norm $p^2$. Hence there exists an element of norm $p^k$ if and only if $k$ is even. In this case, such an element is an associate of $p^{k/2}$. This implies that $|Y(a,p^k)|=u$ if $k$ is even and $|Y(a,p^k)|=0$ if $k$ is odd.
\end{enumerate}

From all the three cases, we have
\[
f(p^k) = \sum_{j=0}^k \left(\frac{D}{p}\right)^j = \left[\one*\left(\frac{D}{\cdot}\right)\right](p^k),
\]
where $*$ denotes the Dirichlet convolution of two arithmetic functions. Hence \eqref{eqn:Yan-general} follows from the multiplicativity of both $f$ and the convolution $\one*\left(\frac{D}{\cdot}\right)$. The formula \eqref{eqn:Yan-3mod4} then follows from the reciprocity law for the Kronecker symbol.
\end{proof}

\section{Counting the number of representations} \label{sec:count-repr}
We now count the number of elements of
\[
X(n,a) = \{(x,y)\in\bZ\times\bZ \ | \ x^2+ay^2=n\}.
\]
An important ingredient is to consider the mapping
\begin{equation} \label{eqn:ph-a-n}
    \ph_{n,a}:X(n,a)\mapsto Y(n,a), \ (x,y)\mapsto x+y\sqrt{-a}.
\end{equation}
It is easily seen that $\ph_{n,a}$ is well-defined and injective. The surjectivity holds if every element of $Y(n,a)$ is of the form $b+c\sqrt{-a}$ for some $b,c\in\bZ$ (i.e.~$b$ and $c$ are not half of odd integers). This is particularly the case if $\qZ{a}=\bZ[\sqrt{-a}]$, i.e.~$a\equiv1,2\pmod{4}$. The only such Heegner numbers are $1$ and $2$. Hence we get the following result:
\begin{thm} \label{thm:Xan-1-2}
For all natural numbers $n$, we have
\[
|X(n,1)| = 4 \sum_{c\mid n} \left(\frac{-4}{c}\right) \quad \text{and} \quad |X(n,2)| = 2 \sum_{c\mid n} \left(\frac{-2}{c}\right).
\]
\end{thm}
\begin{proof}
This follows from the observation above together with Proposition \ref{prop:count-given-norm} and the fact that $(-8|\delta)=(-2|\delta)^3=(-2|\delta)$ for every $\delta\in\bZ$.
\end{proof}

If $a$ is a Heegner number such that $a\geq 3$, then necessarily $a\equiv3\pmod{4}$, which implies that $\qZ{a}=\bZ[\lambda_a]$, where $\lambda_a := \frac{-1+\sqrt{-a}}{2}$. Consequently, the map $\ph_{n,a}$ defined in \eqref{eqn:ph-a-n} may not be surjective. The case $a=3$ leads to the ring of Eisenstein integers, which contains primitive third roots of unity. Hence this case needs to be discussed separately and will be postponed to Section \ref{sec:eisenstein}. Instead, we will discuss first the cases $a=7$, where $2$ splits completely in $\qZ{a}$, and $a\geq 11$, where $2$ remains prime in $\qZ{a}$.

\subsection{The case \texorpdfstring{$a=7$}{a=7}}
The number of representations by the quadratic form $x^2+7y^2$ can be counted as follows:
\begin{thm} \label{thm:Xan-7}
Let $n$ be a natural number.
\begin{enumerate}[label=\textup{(\alph*)}]
\item \label{item:a7-n4} If $4\mid n$, then $\displaystyle |X(n,7)| = 2\displaystyle\sum_{c|\frac{n}{4}} \left ( \frac{c}{7} \right )$.
\item \label{item:a7-n2} If $n$ is even but $4\nmid n$, then $X(n,7)=\emptyset$.
\item \label{item:a7-n1} If $n$ is odd, then $\displaystyle |X(n,7)| = 2\displaystyle\sum_{c|n} \left ( \frac{c}{7} \right )$.
\end{enumerate}
\end{thm}
\begin{proof}
We first treat the case $n$ is even. Observe that for any $(x,y)\in X(7,n)$, we have $x\equiv y\pmod{2}$. Consequently, if $X(7,n)$ contains an element $(x,y)$, then $n=x^2+7y^2\equiv0\pmod{4}$. This proves \ref{item:a7-n2}. In order to prove \ref{item:a7-n4}, observe that if $4\mid n$, there is a bijection
\[
X(n,7)\to Y(\tfrac{n}{4},7), \ (x,y) \mapsto \tfrac{x+y}{2}+y\lambda_7.
\]
This together with Proposition \ref{prop:count-given-norm} proves \ref{item:a7-n4}.

Now we come to the case $n$ is odd and claim that the map $\ph_{n,7}$ from \eqref{eqn:ph-a-n} is bijective. To see the surjectivity, observe that if $z=a+b\lambda_7\in Y(7,n)$, then
\[
n = N(z) = a^2-ab+2b^2 \equiv a(a-b)\pmod{2}.
\]
This implies that $a$ and $ab$ are odd. Hence $b$ is even, i.e.~$b=2y$ for some $y\in \bZ$. Consequently, $z=(a-y)+y\sqrt{-7}=\ph(a-y,y)$. This together with Proposition \ref{prop:count-given-norm} proves \ref{item:a7-n1}.
\end{proof}

\subsection{The case \texorpdfstring{$a\geq11$}{a>=11}}
For the remaining case, we have $\qZ{a}^\times=\{\pm1\}$ and $2$ remains prime in $\qZ{a}$. Hence $\qZ{a}/(2)$ is a field with four residue classes represented by $0,1,\lambda_a$ and $\lambda_a^2\equiv 1+\lambda_a\bmod2$. Consequently, the elements of $X(n,a)$ will be enumerated differently from the case $a=7$. We begin with the following observation:
\begin{lem}
Let $a$ be a Heegner number such that $a\geq 11$ and $p$ be a prime number which splits completely in $\qZ{a}$. Then $p$ can be written as $x^2+ay^2$ for some $x,y\in\bZ$ if and only if its prime factors in $\qZ{a}$ are congruent to $1$ modulo $2$.
\end{lem}
\begin{proof}
Let $\pi=b+c\lambda_a$ be a prime factor of $p$. Then $p=\pi\overline{\pi}$ with $\overline{\pi}=b+c\overline{\lambda_a} = (b-c)-c\lambda_a$. If both $\pi$ and $\overline{\pi}$ are congruent to $1$ modulo $2$, then $2\mid c$, say $c=2v$ for some $v\in\bZ$, implying that $p=N(b+2v\lambda_a)=(b-v)^2+av^2$. Conversely, if $p=x^2+ay^2=(x+y\sqrt{-a})(x-y\sqrt{-a})$ for some $x,y\in\bZ$, then $b+c\lambda_a=\pm x\pm y\sqrt{-a} = \pm x\pm(y+2y\lambda_a)$, implying that $c=\pm2y$, i.e. $\pi,\overline{\pi}\equiv 1\pmod{2}$ as desired.
\end{proof}
\begin{ex}
Consider $a=11$ and $p=5$. Since $5$ is a quadratic residue modulo $11$, it follows that $5$ splits completely in $\qZ{11}$. In fact, $5=(2+\lambda_{11})(2+
\overline{\lambda_{11}})$. Its prime factors in $\qZ{11}$, namely $2+\lambda_{11}$ and $2+\overline{\lambda_{11}} = 1-\lambda_{11}$ are both not congruent to $1$ modulo $2$. This corresponds to the fact that $5$ is not expressible as $x^2+11y^2$ for any $x,y\in\bZ$ as can be easily seen.
\end{ex}
\begin{lem} \label{lem:x2ay2even-a3mod8}
Let $a$ be a square-free natural number congruent to $3$ modulo $8$ and $n$ be an even natural number. Then $\#X(n,a)=\#Y(n,a)$.
\end{lem}
\begin{proof}
Consider the map $\ph_{n,a}$ defined in \eqref{eqn:ph-a-n}. It is easy to see that $\ph_{n,a}$ is well-defined and injective. To see the surjectivity, let $z\in Y(n,a)$. Then $2\mid n=N(z)=z\overline{z}$. Since $2$ remains prime in $\qZ{a}$ (note that this does not require the uniqueness of factorization in $\qZ{a}$), it follows that $2\mid z$ or $2\mid\overline{z}$, but the latter case also implies that $2\mid z$. Hence there are $b,c\in\bZ$ such that $z=2(b+c\lambda_a)=\ph_{n,a}(2b-c,c)$ as desired.
\end{proof}

\begin{thm} \label{thm:Xan-11}
Let $n$ be a natural number and $a$ be a Heegner number such that $a\geq 11$.
\begin{enumerate}[label=\textup{(\alph*)}]
\item \label{item:a11-n2} If $n$ is even, then $\displaystyle |X(n,a)| = |Y(n,a)|$.
\item \label{item:a11-n1} If $n$ is odd, then
\begin{equation} \label{eqn:Xan-11-nodd}
|X(n,a)| = \frac13
\left[1+2\cdot\frac{(\tau(n_q)|3)}{\tau(n_q)}\right]\cdot|Y(n,a)|,
\end{equation}
where $n_q$ denotes the product of all prime factors (including multiplicity) of $n$ which are quadratic residues modulo $a$ but not expressible as $x^2+ay^2$ for any integers $x,y$.
\end{enumerate}
\end{thm}
Note that the results from \cite[Cor.\,1 and 3]{Kaplan2004OnTN} are contained in \ref{item:a11-n2}.

\begin{ex}
Let $a=11$ and $n=437805=3^4\cdot5\cdot23\cdot47$. We see that $47=6^2+11\cdot1^2$ and $3,5,23$ are the prime factors of $n$ which are quadratic residues modulo $11$ but not expressible as $x^2+11y^2$. This implies that $|Y(n,a)|=2\tau(n)=80$ and $n_q=3^4\cdot5\cdot23$, i.e.~$\tau(n_q)=20$. Therefore,
\[
|X(n,a)| = \frac{1}{3}\left[1+2\cdot\frac{(20|3)}{20}\right]\cdot80 = 24
\]
In fact, a computation shows that
\[
X(n,a) = \left\{\begin{array}{c} (\pm78,\pm609),(\pm114,\pm543),(\pm126,\pm513),\\(\pm166,\pm367),(\pm182,\pm271),(\pm198,\pm81)\end{array} \right\}.
\]
\end{ex}

\begin{proof}[Proof of Theorem \ref{thm:Xan-11}]
The case $n$ is even follows from Lemma \ref{lem:x2ay2even-a3mod8}. Now assume that $n$ is odd. Following the proof of loc.~cit., determining $|X(n,a)|$ amounts to counting the number of elements of $Y(n,a)$ of the form $x+y\sqrt{-a}=(x-y)+2y\lambda_a$ for some $x,y\in\bZ$, or equivalently, those congruent to $1$ modulo $2\qZ{a}$. To this end, fix a set $\dP$ of representatives of the association classes of prime elements of $\qZ{a}$ which is stable under conjugation. For each $z\in Y(n,a)$, consider its prime factorization (note that $a$ is necessarily a prime integer and is ramified in $\bQ(\sqrt{-a})/\bQ$):
\begin{equation} \label{eqn:factor-in-qZa}
    z = \ep(\sqrt{-a})^{\delta}\prod_{i=1}^l \Bigl(\pi_{i}^{\alpha_i}\overline{\pi}_{i}^{\alpha'_i}\Bigr) \prod_{j=1}^m \Bigl(\rho_{j}^{\beta_j}\overline{\rho}_{j}^{\beta'_j}\Bigr) \prod_{k=1}^sr_k^{\gamma_k},
\end{equation}
where $\ep\in\qZ{a}^\times$; $(\pi_i,\bar{\pi}_i)$ are  conjugate pairs of primes in $\dP$ dividing prime numbers $p_i$ which can be written as $x^2+ay^2$; $(\rho_j,\overline{\rho}_j)$ are conjugate pairs of primes in $\dP$ dividing prime numbers $q_j$ which are quadratic residues modulo $a$ but not expressible as $x^2+ay^2$ such that $\rho_j\equiv\lambda_a^2$ and hence $\overline{\rho}_j\equiv\lambda_a\bmod{2}$; and $r_k$ are prime numbers that are not quadratic residues modulo $a$. Taking norm and comparing this with the prime factorization of $n$ yields
\[
\delta = v_a(n), \quad \alpha_i+\alpha_i'=v_{p_i}(n), \quad \beta_j+\beta_j'=v_{q_j}(n) \quad \text{and} \quad 2\gamma_k=v_{r_k}(n),
\]
where $v_p(n)$ denotes the $p$-adic valuation of $n$, i.e.~the exponent of the highest power of $p$ that divides $n$. This means that the values of $\delta$ and $\gamma_k$'s are fixed and $0\leq\alpha_i,\alpha_i'\leq v_{p_i}(n)$ for all $i$ and $0\leq\beta_j,\beta_j'\leq v_{q_j}(n)$ for all $j$. Furthermore, reducing \eqref{eqn:factor-in-qZa} modulo $2$ yields
\[
1\equiv z \equiv \prod_j \lambda_a^{2\beta_j+\beta_j'} = \prod_j \lambda_a^{2\beta_j+v_{q_j}(n)-\beta_j} = \lambda_a^{\sum_{j}(v_{q_j}(n)+\beta_j)} \pmod{2}.
\]
This means that, provided that $2\mid v_r(n)$ for all primes $r$ which remain prime in $\qZ{a}$, we have
\begin{equation} \label{eqn:Xan-11-nodd-pre}
|X(n,a)| = 2\prod_{i=1}^l\bigl(v_{p_i}(n)+1\bigr) \cdot \bigl|T(v_{q_1}(n),\ldots,v_{q_m}(n))\bigr|
\end{equation}
where
\begin{equation} \label{eqn:T-v-v}
T(v_1,\ldots,v_m) := \Biggl\{ (b_1,\ldots,b_m)\in\bN_0^m \ \bigg| \ b_j\leq v_j \ \text{and} \ \sum_{j=1}^m (v_j+b_j) \equiv 0 \bmod{3} \Biggr\}.
\end{equation}
To determine the cardinality of $T(v_1,\ldots,v_m)$, observe that it is equal to the constant term of the remainder in the polynomial division of
\[
F_{(v_1,\ldots,v_m)}(x) := \prod_{j=1}^m G_{v_j}(x)
\]
by $x^3-1$, where $G_v(x):=x^v+x^{v+1}+\cdots+x^{2v}$. To compute this, write
\begin{equation} \label{eqn:mod-x3-1}
    F_{(v_1,\ldots,v_m)}(x) = Q(x)(x^3-1)+(A+Bx+Cx^2),
\end{equation}
where $Q(x)\in\bZ[x]$ and $A,B,C\in\bZ$. Denote by $\omega\in\bC$ a primitive third root of unity. Evaluating \eqref{eqn:mod-x3-1} in $x$ at $1,\omega,\omega^2$ and summing all three obtained equations yields
\[
3A = F_{(v_1,\ldots,v_m)}(1) + F_{(v_1,\ldots,v_m)}(\omega) + F_{(v_1,\ldots,v_m)}(\omega^2).
\]
On the other hand, we have $G_v(1)=v+1$ and $G_v(\omega)=G_v(\omega)=r(v)\in\{-1,0,1\}$ such that $v+1\equiv r(v)\pmod{3}$. Note that $r(v)$ is exactly the Legendre symbol of $v+1$ over $3$. This implies that
\begin{equation} \label{eqn:T-cardinality}
    |T(v_1,\ldots,v_m)| = A = \frac{1}{3}\left[V+2\left(\frac{V}{3}\right)\right], \quad \text{where} \ V:=\prod_{j=1}^m(v_j+1).
\end{equation}
Hence \eqref{eqn:Xan-11-nodd} follows from \eqref{eqn:Xan-11-nodd-pre} and \eqref{eqn:T-cardinality} in combination with the observation that if $2\mid v_r(n)$ for all primes $r$ which remain prime in $\qZ{a}$, then
\[
Y(n,a) = 2\prod_{i=1}^l\bigl(v_{p_i}(n)+1\bigr)\prod_{j=1}^m\bigl(v_{q_j}(n)+1\bigr),
\]
which can be deduced from Proposition \ref{prop:count-given-norm}.
\end{proof}

\section{Applications of Eisenstein integers and the Cubic Reciprocity Law} \label{sec:eisenstein}
The special feature of the case $a=3$ is that $\qZ{3}^\times$ is exactly the group of the sixth roots of unity, whereas $\qZ{a}$ for $a>3$ consists of only $1$ and $-1$. Hence the discussion needs to be done separately. In what follows, we will write $\omega:=\lambda_3\in\qZ{3}$. This is a primitive third root of unity and satisfies $1+\omega+\omega^2=0$.

\begin{thm} \label{thm:Xan-3}
Let $n$ be a natural number.
\begin{enumerate}[label=\textup{(\alph*)}]
\item If $n$ is even, then $\displaystyle |X(n,3)| = 6\displaystyle\sum_{c|n} \left ( \frac{c}{3} \right )$.
\item If $n$ is odd, then $\displaystyle |X(n,3)| = 2\displaystyle\sum_{c|n} \left ( \frac{c}{3} \right )$.
\end{enumerate}
\end{thm}

\begin{proof}
We first treat the case $n$ is even and claim that the map $\ph_{n,3}$ from \eqref{eqn:ph-a-n} is bijective. To see the surjectivity, observe that for all $z\in Y(n,a)$, we have $2\mid N(z)=z\overline{z}$. Since $2$ remains prime in $\qZ3$ and $2\mid z$ if and only if $2\mid \overline{z}$, it follows that $2\mid z$, i.e.~$z=2a+2b\omega=\ph_{n,3}(2a-b,b)$ for some $a,b\in\bZ$. Combining this result with Proposition \ref{prop:count-given-norm} yields
\[
|X(n,3)| = |Y(n,3)| = 6\displaystyle\sum_{c|n} \left ( \frac{c}{3} \right )
\]

We now come to the case $n$ is odd. The set $Y(n,a)$ may be partitioned into subsets of the form $\{z,z\omega,z\omega^2\}$, i.e.~orbits under the action of $\{1,\omega,\omega^2\}$ given by the multiplication. Now $\{1,\omega,\omega^2\}$ is a reduced residue system in $\qZ{3}$ modulo $2$. Hence for each $z\in Y(n,a)$, the $z,z\omega,z\omega^2$ are all different modulo $2$ since $\gcd(2,z)=1$. Hence exactly one of them is congruent to $1$, i.e.~of the form $a+2b\omega=\ph(a-b,b)$ for some $a,b\in\bZ$. Combining this result with Proposition \ref{prop:count-given-norm} yields
\[
|X(n,3)| = \frac13|Y(n,3)| = 2\displaystyle\sum_{c|n} \left ( \frac{c}{3} \right ) \qedhere
\]
\end{proof}

We conclude with the case $a=27$, in which also the Cubic Reciprocity Law is involved. To this end, observe first that \cite[Prop.\,9.3.5]{ireland-rosen} each association class of prime elements of $\qZ{3}$ not dividing $3$ contains exactly one \textbf{primary prime}, by which we mean a prime element $\pi\in\qZ{3}$ congruent to $2$ modulo $3$, i.e.~$\pi=a+b\omega$ for some integers $a,b$ such that $a\equiv2$ and $b\equiv0\bmod3$.

\begin{thm}
Let $p$ be a prime number congruent to $1$ modulo $3$. The following are equivalent:
\begin{enumerate}
    \item There are integers $x,y$ such that $p=x^2+27y^2$.
    \item $2$ is a cubic residue modulo $p$.
    \item If $\pi\in\qZ{3}$ is a primary prime dividing $p$, then $\pi\equiv1\pmod{2}$.
\end{enumerate}
\end{thm}

\begin{proof}
\cite[Prop.\,9.6.1-2]{ireland-rosen}
\end{proof}

\begin{thm} \label{thm:Xan-27}
Let $n$ be a natural number.
\begin{enumerate}[label=\textup{(\alph*)}]
\item \label{item:a27-n3} If $3 \mid n$, then $X(a,27)\neq \emptyset$ only if $9 \mid n$. In this case, we have 
\[
|X(n,27)| = |X(\tfrac{n}{9},3)|.
\]
\item \label{item:a27-n2} If $3\nmid n$ but $2\mid n$, then
\[
|X(n,27)| = \frac13|X(n,3)|.
\]
\item \label{item:a27-n61} If $\gcd(6,n)=1$, then
\[
|X(n,27)| = \frac13\left[1+2\cdot\frac{(\tau(n_q)|3)}{\tau(n_q)}\right]\cdot|X(n,3)|,
\]
where $n_q$ denotes the product of all prime factors $q$ (including multiplicity) of $n$ such that $q\equiv1\pmod{3}$ and $2$ is not a cubic residue modulo $q$.
\end{enumerate}
\end{thm}
Note that this result agrees with \cite[Thm.\,6.1]{berkovich-ramanujan}.
\begin{proof}
We begin with the case $3\mid n$. Observe that if $(x,y)\in X(27,n)$, then $3$ divides $n-27y^2=x^2$. Hence $x=3v$ for some $v\in\bZ$. In particular, if $X(27,n)$ contains an element $(x,y)=(3v,y)$, then $n=9v^2+27y^2$, i.e.~$9\mid n$ and $v^2+3y^2=\frac{n}{9}$. This yields a bijection between $X(n,27)$ and $X(\frac{n}{9},3)$, which proves \ref{item:a27-n3}.

In order to treat the case $3\nmid n$, observe first that the map
\begin{equation} \label{eqn:psi-27-3}
\psi:X(n,27)\to Y(n,3), \ (x,y)\mapsto x+3y\sqrt{-3}.
\end{equation}
is well-defined and injective. Furthermore, $z+b\omega\in Y(n,3)$ lies in the image of $\psi$ if and only if $6\mid b$, i.e.~$z\equiv\pm1,\pm2\bmod{6}$

For the case $2\mid n$, i.e.~$\gcd(6,n)=2$, we claim that for each $z\in Y(n,3)$, exactly one of $z,z\omega,z\omega^2$ is in the image of $\psi$. In fact, the condition $3\nmid n$ and $2\mid n$ implies that $(1-\omega)\nmid z$ and $2\mid z$. Consequently, $\gcd(z,6)=\gcd(z,2(1-\omega)^2)=2$. Hence $z$ is congruent to exactly one of the following elements modulo $6$:
\[
\pm2,\pm2\omega,\pm2\omega^2=\mp2\mp2\omega.
\]
This implies that exactly one of $z,z\omega,z\omega^2$ is of the form $a+b\omega$ for some $a,b\in\bZ$ such that $a\equiv\pm2\pmod{6}$ and $6\mid b$. Therefore $|Y(n,3)|=3|X(n,27)|$, which proves \ref{item:a27-n2} in combination with Theorem \ref{thm:Xan-3}.

We now come to the case $\gcd(6,n)=1$. For each $z\in Y(n,3)$, consider its prime factorization
\begin{equation} \label{eqn:factor-eisenstein}
    z = \ep\prod_{i=1}^l \Bigl(\pi_{i}^{\alpha_i}\overline{\pi}_{i}^{\alpha'_i}\Bigr) \prod_{j=1}^m \Bigl(\rho_{j}^{\beta_j}\overline{\rho}_{j}^{\beta'_j}\Bigr) \prod_{k=1}^sr_k^{\gamma_k},
\end{equation}
where $\ep\in\qZ{3}^\times=\{\pm1,\pm\omega,\pm\omega^2\}$; $(\pi_i,\bar{\pi}_i)$ are  conjugate pairs of primary primes dividing prime numbers $p_i$ such that $\pi_i\equiv1\bmod2$; $(\rho_j,\overline{\rho}_j)$ are conjugate pairs of primary primes dividing prime numbers $q_j$ such that $\rho_j\equiv\omega^2$ and hence $\overline{\rho}_j\equiv\omega\bmod{2}$; and $r_k$ are prime numbers that are not quadratic residues modulo $3$. Taking norm and comparing this with the prime factorization of $n$ yields
\[
\alpha_i+\alpha_i'=v_{p_i}(n), \quad \beta_j+\beta_j'=v_{q_j}(n) \quad \text{and} \quad 2\gamma_k=v_{r_k}(n),
\]
i.e.~the values of $\gamma_k$'s are fixed and $0\leq\alpha_i,\alpha_i'\leq v_{p_i}(n)$ for all $i$ and $0\leq\beta_j,\beta_j'\leq v_{q_j}(n)$ for all $j$. Furthermore, reducing \eqref{eqn:factor-eisenstein} modulo $3$ yields
\[
\pm1\equiv z \equiv \ep(-1)^{\sum_i(\alpha_i+\alpha'_i)+\sum_j(\beta_j+\beta_j')+\sum_k\gamma_k} \pmod{3},
\]
which implies that $\ep=\pm1$. Now reducing \eqref{eqn:factor-eisenstein} modulo $2$ yields
\[
1\equiv z \equiv \prod_j \lambda_a^{2\beta_j+\beta_j'} = \prod_j \lambda_a^{2\beta_j+v_{q_j}(n)-\beta_j} = \lambda_a^{\sum_{j}(v_{q_j}(n)+\beta_j)} \pmod{2}.
\]
This means that, provided that $2\mid v_r(n)$ for all primes $r$ which remain prime in $\qZ{3}$, we have
\begin{equation} \label{eqn:Xan-27-n6-pre}
|X(n,a)| = 2\prod_{i=1}^l\bigl(v_{p_i}(n)+1\bigr) \cdot \bigl|T(v_{q_1}(n),\ldots,v_{q_m}(n))\bigr|
\end{equation}
where $T(v_{q_1}(n),\ldots,v_{q_m}(n))$ is as defined in \eqref{eqn:T-v-v} in the proof of Theorem \ref{thm:Xan-11}. Hence a similar argument from the proof of loc.\,cit.~applies here, which proves \ref{item:a27-n61}.
\end{proof}

\begin{rmk}
Contrary to this case, a criterion for a prime integer $p$ to be of the form $x^2+ay^2$ in an explicit form depending on $a$ is not known to the authors. It is only guaranteed by \cite[Thm.9.2]{cox-primes} that such a polynomial criterion exists. Also the case $a=11$ has been discussed in \cite{primex211y2} with an explicit polynomial, but it is unlikely to extend this result to a general case, even for Heegner numbers.
\end{rmk}

\end{document}